\documentclass{amsart}

\usepackage{amssymb}
\usepackage{graphicx}
\usepackage{MnSymbol}

\usepackage{enumerate}

\usepackage{amsmath,amscd}

\usepackage{amsthm}

\title{Artinian groups of large cardinality}

\author{Samuel M. Corson}
\address{Matematika Saila, UPV/EHU, Sarriena S/N, 48940, Leioa - Bizkaia, Spain}
\email{sammyc973@gmail.com}
\author{Saharon Shelah}
\address{Einstein Institute of Mathematics, The Hebrew University of Jerusalem, Jerusalem 91904 Israel}
\address{Department of Mathematics, Rutgers University, Piscataway, NJ 08854 USA}
\email{shelah@math.huji.ac.il}

\theoremstyle{definition}\newtheorem{theorem}{Theorem}
\theoremstyle{definition}
\theoremstyle{definition}

\theoremstyle{definition}
\theoremstyle{definition}

\theoremstyle{definition}\newtheorem{bigtheorem}{Theorem}

\numberwithin{theorem}{section}
\theoremstyle{definition}
\theoremstyle{definition}\newtheorem{proposition}[theorem]{Proposition}
\theoremstyle{definition}\newtheorem{definition}[theorem]{Definition}
\theoremstyle{definition}
\theoremstyle{definition}
\theoremstyle{definition}
\theoremstyle{definition}
\theoremstyle{definition}\newtheorem{lemma}[theorem]{Lemma}
\theoremstyle{definition}
\theoremstyle{definition}
\theoremstyle{definition}
\theoremstyle{definition}
\theoremstyle{definition}\newtheorem{definitions}[theorem]{Definitions}
\theoremstyle{definition}\newtheorem{notation}[theorem]{Notation}

\newcommand{\Po}{\mathcal{P}}

\newcommand{\cl}{\operatorname{cl}}
\newcommand{\fin}{\operatorname{fin}}
\newcommand{\Fr}{\textbf{Fr}}

\newcommand{\Artin}{\textbf{Artin}}
\newcommand{\artin}{\textbf{artin}}
\newcommand{\fr}{\textbf{fr}}

\begin{document}

\thanks{The work of the first author was supported by Basque Government Grant IT1483-22 and Spanish Government Grants PID2019-107444GA-I00 and PID2020-117281GB-I00.  The second author thanks for partially supporting this research Israel Science Foundation (ISF) grant number 2320/23.  Paper number 1255 on Shelah's archive.}

\begin{abstract}
A group is Artinian if there is no infinite strictly descending chain of subgroups.  Ol'shanskii has asked whether there are Artinian groups of arbitrarily large cardinality.  We show that this problem is essentially the same as an analogous question, regarding universal algebras, asked by J\'onsson in the 1960s.  We further show that these problems are the same as the so-called free subset problem.  As a result, one can have a consistent strong negative answer (from a large cardinal assumption) as well as a consistent positive answer.
\end{abstract}

\maketitle

\begin{section}{Introduction}

Recall that an algebraic structure (e.g. group, semigroup, ring) is \emph{Artinian} if every strictly descending chain of substructures is finite.  This notion is also known in the literature as the \emph{descending chain condition} \cite[p. 120]{J} and in the group setting as the \emph{min condition for subgroups} \cite[p. 62]{Ol}.  Finite groups and quasicyclic groups $\mathbb{Z}(p^{\infty})$ are Artinian.  Ol'shanskii's Tarski monsters \cite{Ol0} are very stark examples of infinite Artinian groups in which a strictly descending chain of subgroups has length at most three.  Obtraztsov \cite{Obr1} constructed an Artinian group of cardinality $\aleph_1$.  Artinian groups must be torsion, so an Artinian group is also an Artinian semigroup if one forgets about the inversion function $^{-1}$.  More generally, if $G$ is Artinian and $H \leq G$ then there exists $K \leq H$, of the same cardinality as $H$, with $K$ \emph{J\'{o}nsson} (every proper subgroup of $K$ has cardinality strictly smaller than that of $K$).

Alexander Ol'shanskii asks in Kourovka Notebook Problem 15.70 whether there exist Artinian groups of arbitrarily large cardinality \cite{KhMaz}.  It will be shown that positive and negative answers are set theoretically consistent.  First, the question will be moved to a more general setting.  Recall that an \emph{algebra} $\mathbb{M} = (M, \mathcal{F})$ is a set $M$ together with a \emph{signature} $\mathcal{F}$ which is a collection of finitary operations on $M$.  The cardinality of an algebra refers to the cardinality of the underlying set $M$.  Bjarni J\'onsson has asked at which cardinals there is an Artinian algebra with finite signature \cite[Remarks p. 133]{J}, \cite{Wh}.

Using a versatile embedding theorem of Obraztsov, we show that the questions of Ol'shanskii and of J\'onsson are fundamentally related.

\begin{bigtheorem}\label{equivalence}  Each of the following holds.

\begin{enumerate} 

\item Let $m$ be an odd integer which is sufficiently large and $\kappa$ be an infinite cardinal.  The following are equivalent.

\begin{enumerate}[(a)]

\item There is a simple Artinian group of cardinality $\kappa$ and exponent $m$.

\item There is an Artinian algebra of cardinality $\kappa$ and countable signature.

\end{enumerate}

\item There exist Artinian groups of size $\aleph_n$ for each natural number $n < \omega$, and if there exists an Artinian group of size $\aleph_{\alpha}$ there exists one of size $\aleph_{\alpha + 1}$.
\end{enumerate}
\end{bigtheorem}

For a fixed infinite cardinal $\lambda$ it is easy to see that exactly one of two possibilities can arise.  It could be the case that there exist Artinian algebras of all cardinalities having signature of size at most $\lambda$.  In the other case there exists a cardinal $\kappa$ such that algebras of size $\kappa$ or larger and signature of size at most $\lambda$ are never Artinian, but there are Artinian algebras having signature of size at most $\lambda$ at all cardinals below $\kappa$.  In the latter case we write $\artin(\lambda) = \kappa$ and in the former $\artin(\lambda) = \infty$.  Theorem \ref{equivalence} (1) tells us that there are Artinian groups at all cardinals greater than $0$ and below $\artin(\aleph_0)$, and no Artinian algebras of countable signature at or above $\artin(\aleph_0)$.  So, determining the value of $\artin(\aleph_0)$ simultaneously answers the question of Ol'shanskii and that of J\'onsson.  To establish further results we move into the setting of the free subset problem.

We say a subset $X \subseteq M$ is \emph{free} in the algebra $\mathbb{M} = (M, \mathcal{F})$ if for each $x \in X$ the subalgebra generated by $X \setminus \{x\}$ does not contain the element $x$ \cite{Dev}.  If an algebra has an infinite free subset then it is not Artinian.  Let $\fr(\lambda) = \kappa$ denote the minimal cardinal such that every algebra having size $\kappa$ and at most $\lambda$ functions has an infinite free subset (or write $\fr(\lambda) = \infty$ in case no such $\kappa$ exists).  It is easy to see that $\artin(\lambda) \leq \fr(\lambda)$.  Much of the effort of the current paper goes into showing that equality holds.

\begin{bigtheorem} \label{Artinfree}

For each infinite cardinal $\lambda$ one has $\fr(\lambda) = \artin(\lambda)$.
\end{bigtheorem}

So our questions so far, and all analogous questions regarding larger signatures, are completely answered modulo an understanding of the function $\fr(\cdot)$.  Fortunately, much is known regarding $\fr(\aleph_0)$, and the most extreme imaginable values can arise.  From a result of the second author \cite{Sh1} we can force (using infinitely many measurable cardinals) $\fr(\aleph_0) = \aleph_{\omega}$, and Theorem \ref{equivalence} (2) tells us that this is the smallest conceivable value.  By a result of Koepke \cite{Koepke} we can say even more.

\begin{bigtheorem}\label{failure}  The following are equiconsistent.

\begin{enumerate}

\item ZFC + ``there exists a measurable cardinal''

\item ZFC + ``there are no Artinian groups of cardinality $\aleph_{\omega}$ or higher''

\end{enumerate}
\end{bigtheorem}

On the positive side we can use a result of Devlin and Paris \cite{DevPar} to obtain the following. 

\begin{bigtheorem} \label{constructibleuniv}  It is consistent that there are Artinian groups of every cardinality greater than $0$.
\end{bigtheorem}

For Theorem \ref{constructibleuniv} we use G\"odel's constructible universe $L$.  In $L$ the first nonzero cardinal at which there is no Artinian group would be the first Erd\H{o}s cardinal $E_0$ (see Theorem \ref{constructiblecutoff}).  Since such a cardinal is inaccessible, one can halt the construction of the constructible universe at stage $E_0$ and obtain a model of ZFC + ($V = L$) in which no such cardinal exists.

Regarding groups we summarize: Artinian groups occur unconditionally at the small cardinals $\aleph_n$ where $n < \omega$ (Theorem \ref{equivalence}), can fail at all cardinals beyond (Theorem \ref{failure}), or can exist at all cardinals greater than $0$ (Theorem \ref{constructibleuniv}).  In particular an answer to the questions of Ol'shanskii and J\'onsson is independent of ZFC.

In Section \ref{theequivalence} we prove Theorem \ref{equivalence}.  Section \ref{relationshipArtinfree} we prove Theorem  \ref{Artinfree}.  We also provide a dichotomy for each infinite cardinal $\lambda$: either $\artin(\lambda)$ is of cofinality $\aleph_0$ (as in Theorem \ref{failure} when $\lambda = \aleph_0$) or $\artin(\lambda)$ is weakly inaccessible.  We also prove that at all infinite cardinals below $\fr(\aleph_0)$ there are simple, torsion-free groups having no infinite free subset (Theorem \ref{groupsnoinfinitefree}).  In Section \ref{CD} we give Theorems  \ref{constructibleuniv} and \ref{failure}.

\end{section}

\begin{section}{Proof of Theorem \ref{equivalence}}\label{theequivalence}

Towards proving Theorem \ref{equivalence} we establish some notation and definitions.

\begin{notation}
As is standard, each ordinal is the set of ordinals strictly below itself, so that $4 = \{0, 1, 2, 3\}$ and $\omega + 1 = \{0, 1, 2, \ldots, \omega\}$.  For a set $X$ we let $\Po(X)$ denote the powerset of $X$, $\Po'(X) = \Po(X) \setminus \{\emptyset\}$, and $\Po_{\fin}'(X)$ the collection of nonempty finite subsets of $X$.  If $\mathbb{M} = (M, \mathcal{F})$ is an algebra and $X \subseteq M$ then we write $\cl_{\mathbb{M}}(X)$ for the set of elements of the subalgebra of  $\mathbb{M}$ generated by $X$.  Sometimes it will be convenient to say that a subset $X \subseteq M$ is a subalgebra of $(M, \mathcal{F})$ when we really mean that $(X, \mathcal{F} \upharpoonright X)$ is a subalgebra of $(M, \mathcal{F})$.  If $G$ is a group and $Y$ is a subset of $G$ we will use the conventional notation $\langle Y \rangle$ to denote the subgroup generated by $Y$.
\end{notation}

\begin{definitions}  If $\{G_i\}_{i \in I}$ is a collection of nontrivial groups, where $\{1\} = G_i \cap G_j$ whenever $i \neq j$, we denote $\Omega^1 := \bigcup_{i \in I} G_i$ the \emph{free amalgam} of the collection $\{G_i\}_{i \in I}$.  Let $\Omega = \Omega^1 \setminus \{1\}$.  An \emph{embedding} of $\Omega^1$ to a group $G$ is an injective function $E: \Omega^1 \rightarrow G$ with each restriction $E \upharpoonright G_i$ a homomorphism.  Clearly an embedding $E$ extends to a homomorphism $\phi$ from the free product $*_{i \in I} G_i$ to $G$, and as each $\phi \upharpoonright G_i$ is an isomorphism we can consider each $G_i$ to be a subgroup of $G$ in that situation.
\end{definitions}

\begin{definition}  Let $\{G_i\}_{i \in I}$ be a collection of nontrivial groups, with $\Omega^1$ and $\Omega$ defined as above.  A function $f: \Po'(\Omega) \rightarrow \Po(\Omega)$ is \emph{generating} in case

\begin{enumerate}[(a)]

\item if $X \subseteq G_i$ for some $i \in I$ then $f(X) = \langle X \rangle \setminus \{1\}$;

\item if $X \subseteq \Omega$ is finite with $X \not\subseteq G_i$ for all $i \in I$ then $f(X) = Y$ is a countable subset of $\Omega$ such that $X \subseteq Y$ and if $Z \subseteq Y$ is finite nonempty then $f(Z) \subseteq Y$; and

\item if $X \subseteq \Omega$ is infinite and $X \not\subseteq G_i$ for all $i \in I$ then $f(X) = \bigcup_{Z \in \Po_{\fin}'(X)} f(Z)$.

\end{enumerate}

\end{definition}

It is easy to see that for $X \in \Po'(\Omega)$ we have $f(f(X)) = f(X)$.  If $\mathcal{W}$ is a nonempty  collection of nonempty words in the letters $\Omega$ we let $F(\mathcal{W}) = f(X)$ where $X$ is the set of letters which appear in elements of $\mathcal{W}$.  The following is a special case of a theorem of Obraztsov \cite[Theorem A]{Obr2}.

\begin{proposition}\label{beautifulObr}  Suppose that $\{G_i\}_{i \in I}$ is a collection of nontrivial groups without involutions and that $f$ is a generating function, with $|I| \geq 2$.  Suppose also that $m$ is a sufficiently large odd number (or $m = \infty$).  Then there is a simple group $G$ and an embedding $E: \Omega^1 \rightarrow G$ of the free amalgam which induces a homomorphism $\phi: *_{i \in I} G_i \rightarrow G$ satisfying the following properties:

\begin{enumerate}[(i)]

\item $\phi$ is surjective (so we may consider $\Omega$ as a generating set for $G$);

\item if $h \in G$ is not conjugate in $G$ to an element in one of the groups $G_i$ then $h$ is of order dividing $m$ ($h$ is of infinite order in case $m = \infty$);

\item each subgroup $M$ of $G$ is either cyclic, or conjugate in $G$ to a subgroup of one of the $G_i$, or conjugate in $G$ to a subgroup of form $\langle C \rangle$ where $C = f(X)$ for some $ X \in \Po'(\Omega)$;

\item if $C = f(X)$ for some $X \in \Po'(\Omega)$ and $g \in G$ is given by word $W$ which is of minimal length in the alphabet $\Omega$, then $g \in \langle C \rangle$ if and only if $F(\{W\}) \subseteq C$;

\item if $C = f(X)$ for some $X \in \Po'(\Omega)$, $J = \{i \in I: G_i \cap \langle C \rangle \neq \{1\}\}$, and $W_0 \in W_1^{-1}\langle C \rangle W_1$ where the length of $W_1$ is minimal among elements of $\langle C \rangle W_1$ and $G_j W_1$ for each $j \in J$, then $F(\{W_1\}) \subseteq F(\{W_0\})$.

\end{enumerate}

\end{proposition}

\begin{proof}[Proof of Theorem \ref{equivalence}]
For (1) the implication (a) $\Rightarrow$ (b) is trivial, so we show (b) $\Rightarrow$ (a).  We suppose that $m$ is a sufficiently large number (in the sense of Proposition \ref{beautifulObr}) and that there is a universal algebra, of infinite cardinality $\kappa$ and countable signature, which is Artinian.  Let $\{G_{\alpha}\}_{\alpha < \kappa}$ be a collection of groups where each $G_{\alpha}$ is isomorphic to the cyclic group $\mathbb{Z}/m\mathbb{Z}$ of order $m$ and $G_{\alpha} \cap G_{\beta} = \{1\}$ whenever $\alpha \neq \beta$.  For each $\alpha < \kappa$ let $z_{\alpha}$ generate $G_{\alpha}$.  Let $\Omega^1$ be the free amalgam of the groups $\{G_{\alpha}\}_{\alpha < \kappa}$ and $\Omega = \Omega^1 \setminus \{1\}$.  Define $\tau: \Omega \rightarrow \{z_{\alpha}\}_{\alpha < \kappa}$ by $g \mapsto z_{\alpha}$ where $g \in G_{\alpha} \setminus \{1\}$.

By assumption we let $\mathbb{M} = (\{z_{\alpha}\}_{\alpha < \kappa}, \{j_n\}_{n < \omega})$ be an Artinian algebra of countable signature. Take $\pi: \{z_{\alpha}\}_{\alpha < \kappa} \rightarrow \kappa$ to be the map $z_{\alpha} \mapsto \alpha$.  It is straightforward to check that the function $f: \Po'(\Omega) \rightarrow \Po(\Omega)$ defined by

\[
f(X) = \left\{
\begin{array}{ll}
\langle X \rangle \setminus \{1\}
                                            & \text{if } X \subseteq G_{\alpha} \text{ for some }a < \kappa, \\
(\bigcup_{\alpha \in \pi(\cl_{\mathbb{M}}(\tau(X)))} G_{\alpha}) \setminus \{1\}                                           & \text{otherwise}
\end{array}
\right.
\]

\noindent is generating (see the proof of \cite[Theorem 1.1]{C}).  Take $E: \Omega \rightarrow G$ to be an embedding given by Proposition \ref{beautifulObr}, using the prime $p$ for the number $m$.

We immediately know that $G$ is simple, has exponent $m$ (by part (ii) of the proposition), and is of cardinality $\kappa$ (by part (i) of the proposition and the fact that $E$ is an injection and $\kappa$ is infinite).  We must check that $G$ is Artinian.  For this we first prove the following.

\begin{lemma}\label{gettingout}
Suppose that $C = f(X)$ with $X \in \Po'(\Omega)$ and $X \not\subseteq G_{\alpha}$ for each $\alpha < \kappa$.  If $g \in G \setminus \langle C \rangle$ then for all $\alpha < \kappa$ we have $G_{\alpha} \cap g^{-1}\langle C \rangle g = \{1\}$.  
\end{lemma}

\begin{proof}
By how $f(X)$ is defined and clause (iv) of Proposition \ref{beautifulObr} we see that $G_{\beta} \cap \langle C \rangle \neq \{1\}$ implies $G_{\beta} \subseteq C \subseteq \langle C \rangle$.  So if $J = \{\beta < \kappa \mid G_{\beta} \cap \langle C \rangle \neq \{1\}\}$ then $J = \{\beta < \kappa \mid G_{\beta} \leq \langle C \rangle\}$.  Suppose for contradiction that $h \in G_{\alpha} \cap g^{-1}\langle C \rangle g$  is nontrivial for some $g \in G \setminus \langle C \rangle$.  As $h \in G_{\alpha}\setminus \{1\}$ we know that the minimal word in the letters $\Omega$ representing $h$ is simply the word $W_0$ of length $1$ given by $W_0 \equiv h$.

Take $W_1$ to be a word in $\Omega$ which is of minimal length among all expressions in the letters $\Omega$ representing elements of the right coset $\langle C \rangle g$.  If $\beta \in J$ then as $\langle C \rangle W_1 \supseteq (G_{\beta})W_1$ the word $W_1$ will not be longer than a word representing an element in $(G_{\beta})W_1$.  Then we are in the situation of clause (v) of Proposition \ref{beautifulObr}, so $F(\{W_1\}) \subseteq F(\{W_0\}) = f(\{h\}) = \langle h \rangle \setminus \{1\} \subseteq G_{\alpha}$.  So the element $g_1 \in G$ represented by $W_1$ has $g_1 \in G_{\alpha}$, so in fact $g_1hg_1^{-1} \in (G_{\alpha} \cap \langle C \rangle) \setminus \{1\}$, so $G_{\alpha} \leq \langle C \rangle$.  Then $\langle C \rangle g = \langle C \rangle g_1 = \langle C \rangle$, so $g \in \langle C \rangle$, contrary to our assumption.

\end{proof}

Suppose for contradiction that $\{H_k\}_{k < \omega}$ is a strictly descending chain of subgroups of $G$.  Clearly none of the $H_k$ can be cyclic (or more particularly conjugate to a subgroup of one of the $G_{\alpha}$).  By clause (iii) of Proposition \ref{beautifulObr} we can write $g_0^{-1}H_0g_0 = \langle C_0 \rangle$, where $C_0 = f(X_0)$, $X_0 \subseteq \Omega$, $g_0 \in G$ and $X_0$ is not a subset of any $G_{\alpha}$.  Write $K_0 = \langle C_0 \rangle$ and notice that $K_0 > g_0^{-1}H_1g_0$.  By clause (iii) of Proposition \ref{beautifulObr} there exists some $g_1 \in G$ such that $g_1^{-1}g_0^{-1}H_1g_0g_1 = \langle C_1 \rangle$ where $C_1 = f(X_1)$, $X_1 \subseteq \Omega$, and $X_1$ is not a subset of any $G_{\alpha}$.  Picking $\alpha \in \pi(C_1)$ we see that

\begin{center}

$g_1^{-1}K_0g_1 > g_1^{-1}g_0^{-1}H_1g_0g_1 \geq G_{\alpha}$

\end{center}

\noindent so by Lemma \ref{gettingout} we have $g_1 \in K_0$.  Write $K_1 = \langle C_1 \rangle$ and we have $K_0 > K_1$.  Suppose we have selected elements $g_0, \ldots, g_k$ and subsets $C_0, \ldots, C_k$ and $X_0, \ldots, X_k$ in this way as well as defining $K_0 > \cdots > K_k$.  We have $K_k > g_k^{-1}\cdots g_0^{-1}H_{k + 1}g_0\cdots g_k$.  By clause (iii) of Proposition \ref{beautifulObr} we pick $g_{k + 1} \in G$ such that $$K_{k + 1} := g_{k + 1}^{-1}g_k^{-1}\cdots g_0^{-1}H_{k + 1}g_0\cdots g_kg_{k + 1} = \langle C_{k + 1} \rangle$$ where $C_{k + 1} = f(X_{k + 1})$.  We have $K_k > K_{k+1}$ by applying Lemma \ref{gettingout} again.  We now have a strictly descending chain $\{K_k\}_{k < \omega}$ of subgroups with a more familiar description.

As $X_k \not\subseteq G_{\alpha}$ for each $k < \omega$ and $\alpha <\kappa$ we have $f(X_k) = (\bigcup_{\alpha \in \pi(\cl_{\mathbb{M}}(\tau(X_k)))} G_{\alpha})$, and it is easy to check that $\tau(C_k) = \cl_{\mathbb{M}}(\tau(X_k))$.  Thus if $\tau(C_k) = \tau(C_{k + 1})$ then $C_k = f(X_k) = f(X_{k + 1}) = C_{k + 1}$, and therefore $K_k = \langle C_k \rangle = \langle C_{k + 1} \rangle = K_{k + 1}$, a contradiction.  Then $\tau(C_{k + 1})$ is a proper subalgebra of $\tau(C_k)$ for each $k < \omega$, so $\{\tau(C_k)\}_{k < \omega}$ is a strictly descending chain of subalgebras in $\mathbb{M}$, a contradiction.  Thus (b) implies (a).

We argue why Theorem \ref{equivalence} (2) holds.  Whaley has shown that if there is a finite signature Artinian algebra of cardinality $\aleph_{\alpha}$ then there also exists one of the next higher cardinality $\aleph_{\alpha + 1}$ \cite[Theorem 2.3]{Wh}, and Artinian algebras of finite signature exist at $\aleph_n$ for each $n < \omega$.  Thus we apply the equivalence of (a) and (b).

\end{proof}

\end{section}

\begin{section}{Proof of Theorem \ref{Artinfree}}\label{relationshipArtinfree}

In this section we establish Theorem \ref{Artinfree}.  We begin with some definitions and notation.

\begin{definitions}
If $X$ is a set we write $|X|$ for the cardinality of $X$.  If $\lambda$ is a cardinal number we let $[X]^{\lambda} = \{Y \subseteq X \mid |Y| = \lambda\}$ and $[X]^{< \lambda} = \{Y \subseteq X \mid |Y| < \lambda\}$.  Given an algebra $\mathbb{M} = (M, \mathcal{F})$, a subset $X \subseteq M$ is \emph{free} if for each $x \in X$ we have $x \notin \cl_{\mathbb{M}}(X \setminus \{x\})$, recalling that $\cl_{\mathbb{M}}(Y)$ denotes the set of elements in the subalgebra generated by $Y$.  The notation $\Fr_{\mu}(\kappa, \lambda)$, with $\kappa$, $\lambda$, $\mu$ cardinals and $\kappa$ infinite, means that any algebra of cardinality $\kappa$ and signature of size $\leq \mu$ has a free subset of cardinality $\lambda$ \cite[\textsection 4]{Dev}.  We'll write $\Fr(\kappa, \lambda)$ to mean $\Fr_{\aleph_0}(\kappa, \lambda)$.  As is usual, if $\kappa$ is a cardinal we write $\kappa^+$ for the smallest cardinal which is strictly larger than $\kappa$.
\end{definitions}

It is evident that if $\kappa_1 \geq \kappa_0$, $\lambda_0 \geq \lambda_1$, and $\mu_0 \geq \mu_1$ then $\Fr_{\mu_0}(\kappa_0, \lambda_0)$ implies $\Fr_{\mu_1}(\kappa_1, \lambda_1)$.  Also, $\Fr_{\mu}(\kappa, \lambda)$, with $\lambda > 0$, implies $\kappa > \mu$ (simply by considering the algebra on $\kappa$ whose signature consists of all constant functions). For the negative results we will use the following easy lemma.

\begin{lemma}\label{freetonoArtinian}  If $\Fr_{\mu}(\kappa, \aleph_0)$ then there is no Artinian algebra, with signature of size $\leq \mu$, of cardinality $\kappa$ or higher.
\end{lemma}

\begin{proof}  Suppose the hypotheses and let $\mathbb{M} = (M, \mathcal{F})$ be an algebra, with $|\mathcal{F}| \leq \mu$ and $|M| \geq \kappa$.  Take $Y \subseteq M$ with $|Y| = \kappa$ and let $M' = \cl_{\mathbb{M}}(Y)$.  As $\Fr_{\mu}(\kappa, \aleph_0)$ we know $\mu < \kappa$ and so $|M'| = \kappa$.  Thus we can select $X = \{x_n\}_{n < \omega} \subseteq M'$ which is free in the algebra $(M', \mathcal{F}\upharpoonright M')$.  Letting $M_k = \cl_{\mathbb{M}}(\{x_k, x_{k + 1}, \ldots\})$ we get $\{M_k\}_{k < \omega}$ is a strictly descending chain of subalgebras of $\mathbb{M}$.
\end{proof}

\begin{notation}  If $\mathbb{M}$ is a universal algebra write $\tau(\mathbb{M})$ for the signature, and for $f \in \tau(\mathbb{M})$ we write $\sigma(f)$ for the arity of the function (i.e. $f: M^{\sigma(f)} \rightarrow M$).  For infinite cardinals $\kappa$ and $\lambda$ write $\Artin(\kappa, \lambda)$ if every algebra $\mathbb{M}$ with cardinality $\kappa$ and $|\tau(\mathbb{M})| \leq \lambda$ is not Artinian.  Let $\artin(\lambda) = \min\{\kappa \mid \Artin(\kappa, \lambda)\}$ if this set is nonempty, else write $\artin(\lambda) = \infty$.  Let $\fr(\lambda) = \min \{\kappa \mid \Fr_{\lambda}(\kappa, \aleph_0)\}$ provided this set is not empty and write $\fr(\lambda) = \infty$ otherwise.
\end{notation}

\begin{definitions}  Recall that the \emph{cofinality} of an ordinal $\alpha$ is the smallest cardinality of an unbounded subset of $\alpha$.  A cardinal $\kappa$ is

\begin{itemize}

\item \emph{regular} if its cofinality is equal to $\kappa$,    

\item \emph{limit} if it is the supremum of the cardinals strictly below $\kappa$,

\item \emph{weakly inaccessible} if it is both regular and limit.
\end{itemize}

\end{definitions}

\begin{lemma}\label{Niceinequality}  Suppose $\lambda_1$ is an infinite cardinal.  If $\lambda_1 < \lambda_2 < \kappa = \artin(\lambda_1)$ then $\artin(\lambda_2) = \artin(\lambda_1) = \kappa$.
\end{lemma}

\begin{proof}
That $\artin(\lambda_2) \geq \artin(\lambda_1) = \kappa$ is clear.  For the other inequality we suppose for contradiction that $\kappa < \artin(\lambda_2)$.  Let $\mathbb{N}_2$ be an algebra which is Artinian on the set $\lambda_2$ with $|\tau(\mathbb{N}_2)| = \lambda_1$ and for every $u \in [\lambda_2]^{< \omega}$ there is some $\alpha < \lambda_2$ such that $\cl_{\mathbb{N}_2}(\{\alpha\}) \supseteq u$ (this latter condition can be added by taking $H: \lambda_2 \rightarrow [\lambda_2]^{< \omega}$ to be a surjection and adding functions $j_n: \lambda_2^n \rightarrow \lambda_2$ for each $0 < n < \omega$ so that $j_n(\alpha, \ldots, \alpha) = \alpha_n$ whenever $\alpha_1 < \ldots < \alpha_{\ell}$ are the elements of $H(\alpha)$ and $1 \leq n \leq \ell$).  Let $\mathbb{N}_1$ be Artinian with underlying set $\kappa$ and $|\tau(\mathbb{N}_1)| = \lambda_2$.

\noindent $(*)$  Let $\mathbb{M}$ be an algebra on $\kappa$ with $|\tau(\mathbb{M})| = \lambda_1$ such that

\begin{enumerate}[(a)]

\item $(\beta < \kappa \wedge u \in [\kappa]^{< \omega} \wedge \beta \in \cl_{\mathbb{N}_1}(u)) \Rightarrow (\exists \alpha < \lambda_2)(\beta \in \cl_{\mathbb{M}}(u \cup \{\alpha\}))$;

\item $(\forall f \in \tau(\mathbb{N}_2))(\exists f' \in \tau(\mathbb{M}))(\sigma(f) = \sigma(f') \wedge f' \upharpoonright \lambda_2^{\sigma(f)} = f)$;

\item $(\beta < \kappa \wedge u \in [\kappa]^{< \omega} \wedge \beta \in \cl_{\mathbb{N}_1}(u)) \Rightarrow (\exists \alpha < \lambda_2)(\beta \in \cl_{\mathbb{M}}(u \cup \{\alpha\}) \wedge \alpha \in \cl_{\mathbb{M}}(u \cup \{\beta\}))$.

\end{enumerate}

Such an $\mathbb{M}$ is easily obtained but we elaborate.  It is clear that (c) $\Rightarrow$ (a), but the properties will be added in stages.  First extend arbitrarily each $f \in \sigma(\mathbb{N}_2)$ to have domain $\kappa^{\sigma(f)}$ and include such an extension in $\tau(\mathbb{M})$, so that (b) is satisfied.  To satisfy (a) we enumerate $\{f_{\gamma}\}_{\gamma < \lambda_2} = \tau(\mathbb{N}_2)$ and for each $n < \omega$ we add a function $h_n: \kappa^{n + 1} \rightarrow \kappa$ such that $h_n(x_0, \ldots, x_{n - 1}, \alpha) = f_{\alpha}(x_0, \ldots, x_{n - 1})$ whenever $\sigma(f_{\alpha}) = n$.  So, if $\beta < \kappa$, $u \in [\kappa]^{< \omega}$ and $\beta \in \cl_{\mathbb{N}_1}(u)$ we select $v \in [\lambda_2]^{< \omega}$ such that $\beta$ is obtained using the functions $\{f_i\}_{i \in v}$ from $u$, and take $\alpha < \lambda_2$ so that $\cl_{\mathbb{N}_2}(\{\alpha\}) \supseteq v$.  Now clearly $\beta \in \cl_{\mathbb{M}}(u \cup \{\alpha\})$ and (a) holds.  To get (c), for each $n < \omega$ we add a function $h_n': \kappa^n \rightarrow \lambda_2$ such that $h_n'(x_0, \ldots, x_{n - 1}) = \alpha$ when $x_{n-1} \in \cl_{\mathbb{N}_1}(\{x_0, \ldots, x_{n - 2}\})$ and $\alpha \in \lambda_2$ is chosen such that $x_{n - 1} \in \cl_{\mathbb{M}}\{x_0, \ldots, x_{n - 2}, \alpha\}$ (such an $\alpha$ exists by (a)), and is defined arbitrarily otherwise.  Now we have (c), so $(*)$ holds.

We will show that this algebra $\mathbb{M}$ is Artinian to obtain the contradiction.  Suppose that $\{M_n\}_{n \in \omega}$ is a strictly descending chain of subalgebras.  Letting $N_{2, n} = M_n \cap \lambda_2$ we know $N_{2, n} \supseteq N_{2, n+1}$ and by (b) each $N_{2, n}$ is a subalgebra of $\mathbb{N}_2$.  As $\mathbb{N}_2$ is Artinian, the sequence $\{N_{2, n}\}_{n < \omega}$ must eventually stabilize, say $n \geq n(*)$ implies $N_{2, n} = N_{2, n(*)}$.

For each $m < \omega$ select $\alpha_m \in M_{m+1} \setminus M_m$.  As $\mathbb{N}_1$ is Artinian there exists some $k \in [n(*), \omega)$  such that $\alpha_k \in \cl_{\mathbb{N}_1}(\{\alpha_j\}_{j > k})$, so select $\ell \in [k + 1, \omega)$ such that $\alpha_k \in \cl_{\mathbb{N}_1}(\{\alpha_{k + 1}, \ldots, \alpha_{\ell}\})$.  By (c) there exists $\gamma < \lambda_2$ such that

\begin{center}
$\gamma \in \cl_{\mathbb{M}}(\{\alpha_{k + 1}, \ldots, \alpha_{\ell}\} \cup \{\alpha_k\}) \wedge \alpha_{k} \in \cl_{\mathbb{M}}(\{\alpha_{k + 1}, \ldots, \alpha_{\ell}\} \cup \{\gamma\}).$
\end{center}

\noindent Therefore $\gamma \in M_k \cap \lambda_2$, so $\gamma \in M_{k + 1}$ because $k \geq n(*)$.  But now $$\alpha_k \in \cl_{\mathbb{M}}(\{\alpha_{k + 1}, \ldots, \alpha_{\ell}\} \cup \{\gamma\}) \subseteq \cl_{\mathbb{M}}(M_{k + 1}) = M_{k + 1}$$ which is a  contradiction.
\end{proof}

The next lemma generalizes \cite[Theorem 2.3]{Wh} and uses the same strategy of argument.

\begin{lemma}\label{successor}  Suppose $\lambda$ and $\kappa$ are infinite cardinals.  If $\kappa < \artin(\lambda)$ then $\kappa^+ < \artin(\lambda)$.
\end{lemma}

\begin{proof}  Suppose that $\mathbb{N}_1 = (\kappa, \{f_{\alpha}\}_{\alpha < \lambda})$ is an Artinian algebra.  For each $\alpha < \lambda$ let $\overline{f}_{\alpha} : (\kappa^+)^{\sigma(f_{\alpha})} \rightarrow \kappa^+$ be such that $\overline{f}_{\alpha} \upharpoonright \kappa^{\sigma(f_{\alpha})} = f_{\alpha}$.  For each $\kappa \leq \gamma < \kappa^+$ we let $h_{\gamma}: \gamma \rightarrow \kappa$ be a bijection and define $h: (\kappa^+)^2 \rightarrow \kappa^+$ by

\[
h(\beta, \gamma) = \left\{
\begin{array}{ll}
h_{\beta}(\gamma)		& \text{if } \kappa \leq \beta \text{ and }\gamma < \beta, \\
h_{\gamma}^{-1}(\beta)	& \text{if } \beta < \kappa \leq \gamma,\\
0 				& \text{otherwise.}
\end{array}
\right.
\]

\noindent Let $\mathbb{M} = (\kappa^+, \{\overline{f}_{\alpha}\}_{\alpha < \lambda} \cup \{h\})$.  It is clear that if $K$ is a subalgebra of $\mathbb{M}$ then $K \cap \kappa$ is also a subalgebra both of $\mathbb{N}_1$ and of $\mathbb{M}$.  Also, if $\kappa \leq \beta$ and $\beta$ is an element of subalgebra $K$ then $K \cap \kappa = h_{\beta}(K \cap \beta)$.

Showing that $\mathbb{M}$ is Artinian is now done precisely as in \cite{Wh}.  One checks that if $K_0$ is a proper subalgebra of $\mathbb{M}$ and $K_1$ is a proper subalgebra of $K_0$ then at least one of the following holds

\begin{enumerate}[(i)]

\item $K_1 \cap \kappa$ is a proper subalgebra of $K_0 \cap \kappa$;

\item $\sup K_1 < \sup K_0$;

\item $\sup K_0 \in K_0 \setminus K_1$.

\end{enumerate}

\noindent That $\mathbb{M}$ is Artinian follows immediately.

\end{proof}

\begin{lemma} \label{limitcardinal}  Suppose $\lambda$ is an infinite cardinal.  Then $\artin(\lambda)$ is a limit cardinal greater than $\lambda$.
\end{lemma}

\begin{proof}
Suppose for contradiction that $\artin(\lambda) = \mu^+$.  Then $\lambda \leq \mu < \artin(\lambda)$ so $\artin(\mu) = \artin(\lambda)$ (using Lemma \ref{Niceinequality} in case $\lambda < \mu$), but by Lemma \ref{successor} $\mu^+ < \artin(\mu) = \mu^+$, contradiction.  To see that $\lambda < \artin(\lambda)$ we consider the algebra $\mathbb{M} = (\lambda, \{f_{\alpha}\}_{\alpha < \lambda})$ where $f_{\alpha}: \lambda \rightarrow \{\alpha\}$.

\end{proof}

\begin{theorem}\label{cofinality}  If $\lambda$ is an infinite cardinal then $\artin(\lambda) > \lambda$, and either has cofinality $\aleph_0$ or is weakly inaccessible.
\end{theorem}

\begin{proof}
We know that $\artin(\lambda)$ is a limit cardinal greater than $\lambda$ by Lemma \ref{limitcardinal}.  Let $\artin(\lambda) = \kappa$ and suppose for contradiction that $\kappa$ is singular of uncountable cofinality.  Let $\{\kappa_{\alpha}\}_{\alpha < \Theta}$ be an increasing set of cardinals with $\Theta < \kappa$ being the cofinality of $\kappa$ and $\kappa = \sup\{\kappa_{\alpha}\}_{\alpha < \Theta}$.  For each $\alpha < \Theta$ let $\mathbb{N}_{\alpha} = (\kappa_{\alpha} , \{f_{\beta, \alpha}\}_{\beta < \lambda})$ be an Artinian algebra.  For each ordered pair $(\beta, \alpha) \in \lambda \times \Theta$ we take $\overline{f}_{\beta, \alpha}: \kappa^{\sigma(f_{\beta, \alpha})} \rightarrow \kappa$ to be such that $\overline{f}_{\beta, \alpha} \upharpoonright \kappa_{\alpha}^{\sigma(f_{\beta, \alpha})} = f_{\beta, \alpha}$ and $\overline{f}_{\beta, \alpha}(\kappa_{\gamma}^{\sigma(f_{\beta, \alpha})}) \subseteq \kappa_{\gamma}$ for $\beta \leq \gamma < \Theta$.  Letting $\mathbb{M} = (\kappa, \{\overline{f}_{\beta, \alpha}\}_{\beta < \lambda, \alpha < \Theta})$ we will show that $\mathbb{M}$ is Artinian, and since $|\tau(\mathbb{M})| = \lambda \cdot \Theta < \kappa$ this will contradict Lemma \ref{Niceinequality}.

Suppose for contradiction that $\{M_n\}_{n < \omega}$ is a strictly decreasing chain of subalgebras of $\mathbb{M}$.  It is clear that for each $\alpha < \Theta$ and $n < \omega$ the set $M_n \cap \kappa_{\alpha}$ is a subalgebra of $\mathbb{N}_{\alpha}$.  Then for each $\alpha < \Theta$ there is a minimal $n_{\alpha} < \omega$ such that if $n_{\alpha} \leq n < \omega$ then $M_n \cap \kappa_{\alpha} = M_{n_{\alpha}} \cap \kappa$.  Note as well that the function $\alpha \mapsto n_{\alpha}$ is nondecreasing, and since $\Theta$ is an uncountable regular cardinal this function must eventually stabilize, say $n_{\alpha} \leq N < \omega$ for all $\alpha < \Theta$.  Then $M_n = M_N$ for all $n \geq N$, a contradiction.
\end{proof}

\begin{lemma}\label{countablecofinality} If $\lambda$ is an infinite cardinal such that $\artin(\lambda)$ has cofinality $\aleph_0$ then $\artin(\lambda) = \fr(\lambda)$.
\end{lemma}

\begin{proof}
We already know that $\artin(\lambda) \leq \fr(\lambda)$ by Lemma \ref{freetonoArtinian}.   Suppose that $\artin(\lambda) = \kappa$ has cofinality $\aleph_0$.  Take $\kappa = \sup\{\kappa_n\}_{n < \omega}$ with $\lambda < \kappa_n < \kappa_{n + 1}$.  Letting $\mathbb{M}_* = (\kappa, \{g_{\alpha}\}_{\alpha < \lambda})$ be an arbitrary algebra, we shall find an infinite free subset.  For each $n < \omega$ we take $\mathbb{N}_n = (\kappa_n, \{f_{\alpha, n}\}_{\alpha < \lambda})$ to be Artinian and such that each $u \in [\kappa_n]^{< \omega}$ is included in the closure of a singleton (this extra condition can be added as in Lemma \ref{Niceinequality}).  Take $\overline{f}_{\alpha, n}$ to be an extension of $f_{\alpha, n}$ to all of $\kappa$.  Let $\mathbb{M}_0 = (\kappa, \{g_{\alpha}\}_{\alpha < \lambda} \cup \{\overline{f}_{\alpha, n}\}_{\alpha < \lambda, n < \omega})$.  Having defined $\mathbb{M}_j$, $j < \omega$, we define $\mathbb{M}_{j + 1}$ by adding to $\tau(\mathbb{M}_j)$ functions $\{h_{r, s}\}_{r, s < \omega}$ with $h_{r, s} : \kappa^r \rightarrow \kappa_s$ such that $h_{r, s}(x_0, \ldots, x_{r - 1}) = \alpha$ when $x_{r - 1} \in \cl_{\mathbb{M}_j}(\{x_0, \ldots, x_{r - 2}\} \cup \kappa_s)$ and $\alpha \in \kappa_s$ is chosen minimally such that $\cl_{\mathbb{M}_j}(\{x_0, \ldots, x_{r - 2}\} \cup \{\alpha\})$ (and $h_{r, s}$ is defined arbitrarily otherwise).  Take $\mathbb{M}$ to be the algebra with $\tau(\mathbb{M}) = \bigcup_{j < \omega} \tau(\mathbb{M}_j)$, so $|\tau(\mathbb{M})| \leq \lambda$.  Notice that by construction

\begin{enumerate}

\item[$(*)_1$] if $u \in [\kappa]^{< \omega}$ and $\beta \in \cl_{\mathbb{M}}(u \cup \kappa_{n})$ then for some $\gamma < \kappa_n$ we have $\beta \in \cl_{\mathbb{M}}(u \cup \{\gamma\}) \wedge \gamma \in \cl_{\mathbb{M}}(u \cup \{\beta\})$.

\end{enumerate} 

By our choice of $\kappa$ there exists a strictly descending chain $\{M_n\}_{n \in \omega}$ of subalgebras in $\mathbb{M}$.  For a fixed $n \in \omega$ the sequence of ordinals $\{\min(M_n \setminus M_k)\}_{k \in \omega \setminus (n + 1)}$ is nonincreasing and must therefore stabilize.  By passing to a subsequence, we may assume without loss of generality that 

\begin{enumerate}

\item[$(*)_2$]  $\beta_n = \min(M_n \setminus M_{n + 1}) = \min(M_n \setminus M_k)$ for every $k \geq n + 1$

\end{enumerate}

\noindent and we therefore have that

\begin{enumerate}

\item[$(*)_3$] $\beta_n < \beta_{n + 1}$ for all $n \in \omega$

\end{enumerate}

\noindent since $\beta_n = \min(M_n \setminus M_{n + 1}) = \min(M_n \setminus M_{n + 2})$ and $\beta_{n + 1} \neq \beta_n$ because $ \beta_n \in M_n \setminus M_{n + 1}$ and $\beta_{n + 1} \in M_{n + 1}$.  Since each $\mathbb{N}_k$ is Artinian, and by how the $\beta_n$ were selected, we also know

\begin{enumerate}

\item[$(*)_4$]  for each $k < \omega$ there exists $n < \omega$ such that $\beta_n > \kappa_k$; and

\item[$(*)_5$] $m < k < \omega$ implies $M_k \cap \beta_m = M_m \cap \beta_m$.

\end{enumerate}

Consider the following statement.

$\boxplus$  If $m, \ell, n$ satisfy
\begin{enumerate}[\hspace{1cm}(i)]
\item $\beta_m \in [\kappa_{\ell}, \kappa_{\ell + 1})$;

\item $\beta_n \geq \kappa_{\ell + 1}$; and

\item $\beta_0, \ldots, \beta_{m - 1} < \kappa_{\ell}$

\end{enumerate}
 
then $\beta_m \notin \cl_{\mathbb{M}}(\{\beta_i \mid i < m \text{ or }i \geq n\})$.

\noindent It is clear that $\boxplus$ implies the existence of an infinite free subset, for we can select an infinite subset $\{\gamma_j\}_{j < \omega}$ of the $\{\beta_n\}_{n < \omega}$ such that for every $r < \omega$ the set $[\kappa_r, \kappa_{r + 1}) \cap \{\gamma_j\}_{j < \omega}$ has at most one element (by $(*)_3$ and $(*)_4$).  This selection $\{\gamma_j\}_{j < \omega}$ will be an infinite free subset.

So, suppose for contradiction that the assumptions of $\boxplus$ hold but the conclusion fails.  Then $\beta_m \in \cl{\mathbb{M}}(\{\beta_i \mid i < m \text{ or }i \geq n\})$, so by $(*)_1$ there exists a $\gamma < \kappa_{\ell}$ such that $$\beta_m \in \cl_{\mathbb{M}}(\{\beta_i \mid i \geq n\} \cup \{\gamma\}) \wedge \gamma \in \cl_{\mathbb{M}}(\{\beta_i \mid i \geq n\} \cup \{\beta_m\}).$$  Thus $\gamma \in M_m \cap \kappa_{\ell} \subseteq M_m \cap \beta_m$, so by $(*)_5$ we know $\gamma \in M_n$.  But then $\beta_m \in M_n$, which is a contradiction.  Thus $\boxplus$ is true and $\mathbb{M}$ has an infinite free subset, so $\mathbb{M}_*$ also has an infinite free subset.
\end{proof}

\begin{proof}[Proof of Theorem \ref{Artinfree}]  We already know that $\artin(\lambda) \leq \fr(\lambda)$.  Suppose for contradiction that $\artin(\lambda)  = \kappa < \fr(\lambda)$.  By Lemma \ref{countablecofinality} we know $\kappa$ is regular (and even weakly inaccessible).  By Lemmas \ref{Niceinequality} and \ref{successor} we have that $\artin(\lambda) = \artin(\lambda^+)$, and clearly $\fr(\lambda) \leq \fr(\lambda^+)$, so therefore we may assume that $\lambda$ is uncountable.  For each $\beta < \kappa$ we have $|\beta| < \artin(\lambda)$ and therefore let $\mathbb{N}_{\beta} = (\beta, \{p_{\zeta, \beta}\}_{\zeta < \lambda})$ be an Artinian algebra on the set $\beta$.

We define a sequence of algebras $\mathbb{M}_{\epsilon}$ for $\epsilon \leq \aleph_1$ on $\kappa$ such that $\tau(\mathbb{M}_{\epsilon_0}) \subseteq \tau(\mathbb{M}_{\epsilon_1})$, $\mathbb{M}_{\epsilon_0} = \mathbb{M}_{\epsilon_1} \upharpoonright \tau(\mathbb{M}_{\epsilon_0})$ and $|\tau(\mathbb{M}_{\epsilon_1})| \leq \lambda$ for all $\epsilon_0 \leq \epsilon_1 \leq \aleph_1$.  We will begin with some auxilliary constructions.  Take $\mathbb{M}_{-2} = (\kappa, \mathcal{G}_0)$ to be an algebra such that $|\tau(\mathcal{G}_0)| \leq \lambda$ and having no infinite free subset.  Take $\mathbb{M}_{-1} = (\kappa, \mathcal{G}_1)$ to be an algebra such that $\mathcal{G}_1 \supseteq \mathcal{G}_0$, $|\mathcal{G}_1| \leq \lambda$, and for each $u \in [\kappa]^{< \omega}$ there is $\alpha < \kappa$ such that $u \subseteq \cl_{\mathbb{M}_{-1}}(\{\alpha\})$ (see the proof of Lemma \ref{Niceinequality}).  Finally, for each $0 < n < \omega$ and $\zeta < \lambda$ let $g_{n, \zeta}: \kappa^n \rightarrow \kappa$ satisfy $g_n(\beta_0, \ldots, \beta_{n-1}) = p_{\zeta, \beta_0}(\beta_1, \ldots, \beta_{n-1})$ if $n - 1 = \sigma(p_{\zeta, \beta_0})$, and let $\mathbb{M}_0 = (\kappa, \mathcal{G}_{-1} \cup \{g_{n, \zeta}\}_{0 < n < \omega, \zeta < \lambda})$.

Suppose $\tau(\mathbb{M}_{\epsilon})$ has been defined for all $\epsilon < \delta < \aleph_1$.  If $\delta = \epsilon + 1$ for some $\epsilon < \aleph_1$ then for each $0 < r < \omega$ we define function $h_{r, 0}^{\mathbb{M}_{\epsilon + 1}}: \kappa^r \rightarrow \kappa$ given by $h_{r, 0}^{\mathbb{M}_{\epsilon + 1}}(\beta_0, \ldots, \beta_r) = \alpha$ if $\beta_0 \in \cl_{\mathbb{M}_{\epsilon}}(\{\beta_1, \ldots, \beta_{n - 1}\} \cup \{\alpha\})$, and $\alpha$ is minimal such (note that $\alpha \leq \beta_0$).  For each $r < \omega$ we also define function $h_{r, 1}^{\mathbb{M}_{\epsilon + 1}}: \kappa^r \rightarrow \kappa$ by $h_{r, 0}^{\mathbb{M}_{\epsilon + 1}}(\beta_0, \ldots, \beta_{r - 1}) = \alpha$ where $\alpha < \kappa$ is minimal such that $\{\beta_0, \ldots, \beta_{r - 1}\} \subseteq \cl_{\mathbb{M}_{\epsilon}}(\{\alpha\})$ (such an $\alpha$ exists by how $\mathbb{M}_{-1}$ was defined).  Let $\tau(\mathbb{M}_{\epsilon + 1}) = \tau(\mathbb{M}_{\epsilon}) \cup \{h_{r, 0}^{\mathbb{M}_{\epsilon + 1}}\}_{0 < r < \omega} \cup \{h_{r, 1}^{\mathbb{M}_{\epsilon + 1}}\}_{r < \omega}$.  It is clear that $\tau(\mathbb{M}_{\epsilon}) \subseteq \tau(\mathbb{M}_{\delta})$ and $|\tau(\mathbb{M}_{\delta})| \leq \lambda + \aleph_0 \leq \lambda$.  

If $\delta$ is a limit ordinal let $\tau(\mathbb{M}_{\delta}) = \bigcup_{\epsilon < \delta} \tau(\mathbb{M}_{\delta})$ (using the fact that $\lambda$ is uncountable it is clear that $|\tau(\mathbb{M}_{\delta})| \leq \lambda$ even if $\delta = \aleph_1$).

Taking $\mathbb{M} = \mathbb{M}_{\aleph_1}$, we know that $\mathbb{M}$ has no infinite free subset since $\tau(\mathbb{M}) \supseteq \tau(\mathbb{M}_0)$.  We shall show that $\mathbb{M}$ is in fact Artinian, which will contradict $\kappa = \artin(\lambda)$.  To check that the algebra is Artinian, it obviously suffices to check that for every infinite sequence $\{\alpha_n\}_{n < \omega}$ the sequence of generated algebras $\{\cl_{\mathbb{M}}(\{\alpha_{n}\}_{n \geq m})\}_{m < \omega}$ eventually stabilizes (if $\{N_n\}_{n < \omega}$ is a strictly descending chain of subalgebras then we can select $\alpha_n \in N_n \setminus N_{n + 1}$ and check such a sequence).

Suppose that $\{\alpha_n\}_{n < \omega}$ is a sequence such that $\cl_{\mathbb{M}}(\{\alpha_n\}_{n \geq m})$ strictly includes $\cl_{\mathbb{M}}(\{\alpha_n\}_{n \geq m + 1})$ for each $m < \omega$, with $\sup\{\alpha_n\}_{n < \omega}$ minimal.  We may assume without loss of generality that $\{\alpha_n\}_{n < \omega}$ is strictly increasing.

We point out that $\gamma \in \cl_{\mathbb{M}}\{\alpha_n\}_{n < \omega}$ implies $\gamma < \sup\{\alpha_n\}_{n < \omega}$.  To see this, suppose $\beta \in \cl_{\mathbb{M}}\{\alpha_n\}_{n < \omega}$ is such that $\beta > \alpha_n$ for all $n < \omega$.  By how $\mathbb{M}_0$ is defined, we see that for each $m < \omega$ the intersection $\cl_{\mathbb{M}}(\{\alpha_n\}_{n \geq m}) \cap \beta$ is a subalgebra of the Artinian algebra $\mathbb{N}_\beta$.  Then for some $m < \omega$ we get $\alpha_m \in \cl_{\mathbb{M}}(\{\alpha_n\}_{n \geq m + 1}) \cap \beta$, a contradiction.

Notice as well that $\sup \cl_{\mathbb{M}}\{\alpha_n\}_{n < \omega} < \kappa$ since $\kappa$ is an uncountable regular cardinal.

We claim that there exist $u, A \subseteq \omega$ such that

\begin{enumerate}[(i)]

\item $u$ is finite nonempty;

\item $A \subseteq \omega \setminus (\max(u) + 1)$ is infinite;

\item for each $n \in A$ there exists $n < m \in A$ such that $\alpha_n \in \cl_{\mathbb{M}}(\{\alpha_{\ell} \mid \ell \in (n, m) \cap A \text{ or }\ell \in u\})$.

\end{enumerate}

\noindent To see this, suppose such $u, A$ do not exist.  We define a strictly increasing sequence of natural numbers $\{n_i\}_{i < \omega}$ of natural numbers and a chain $A_0 \supseteq A_1 \supseteq \cdots$ of infinite subsets of $\omega$.  Let $n_0 = 0$ and $A_0 = \omega \setminus \{0\}$.  Then $u = \{n_0\}$ and $A = A_0$ does not satisfy (iii) and we can select $n_1 \in A_0$ such that $$\alpha_{n_1} \notin \cl_{\mathbb{M}}(\{\alpha_{\ell} \mid \ell = n_0 \text{ or }\ell \in A_0 \setminus (n_1 + 1)\}).$$  Let $A_1 = A_0 \setminus (n_1 + 1)$.  Supposing we have already selected $n_0 < \cdots < n_i$ and $A_0 \supseteq \cdots \supseteq A_i$, we know that $u = \{n_0, \ldots, n_i\}$ and $A = A_i$ fail to satisfy (iii) so we select $n_{i + 1} \in A_i$ witnessing this and let $A_{i + 1} = A_i \setminus (n_{i + 1} + 1)$.  It is now easy to check that $\{\alpha_{n_i}\}_{i < \omega}$ is a free set in $\mathbb{M}$, giving a contradiction.  Therefore there exist $u$ and $A$ satisfying (i), (ii), (iii).

We let $u' = \{\alpha_{\ell}\}_{\ell \in u} = \{\beta_0, \ldots, \beta_{j - 1}\}$ with $\beta_i < \beta_{i + 1}$ for all $i < j$.  Clearly there exists some $\epsilon < \aleph_1$ such that for every finite $\underline{u} \subseteq \omega$ and $n < \omega$ we have $\alpha_n \in \cl_{\mathbb{M}}(\{\alpha_{\ell}\}_{\ell \in \underline{u}})$ implies $\alpha_n \in \cl_{\mathbb{M}_{\epsilon}}(\{\alpha_{\ell}\}_{\ell \in \underline{u}})$.  Let $\overline{\gamma} = h_{|u'|, 1}^{\mathbb{M}_{\epsilon + 1}}(\beta_0, \ldots, \beta_{j - 1})$.  As $\overline{\gamma} \in \cl_{\mathbb{M}}(\{\alpha_n\}_{n < \omega})$ we have already seen that $\overline{\gamma} < \sup\{\alpha_n\}_{n < \omega}$.

We can assume that $A \cap (\overline{\gamma} + 1) = \emptyset$, by replacing the original $A$ by $A \setminus (\overline{\gamma} + 1)$.  Note that the strictly increasing sequence $\{\alpha_j'\}_{j < \omega}$ such that $\{\alpha_j'\}_{j < \omega} = \{\alpha_n\}_{n \in A}$ has $\alpha_m' \notin \cl_{\mathbb{M}}(\{\alpha_j'\}_{j \geq m+ 1})$ for each $m < \omega$.

Define a strictly increasing sequence $n_0 < n_1 < \cdots$ in $\omega$ by letting $n_0 = 0$ and if $n_i$ has already been defined we select $n_{i + 1} > n_i$ such that $$\alpha_{n_i}' \in \cl_{\mathbb{M}}(u' \cup \{\alpha_{n_i + 1}', \ldots, \alpha_{n_{i + 1} - 1}'\}).$$  Let $\gamma_i = h_{n_{i + 1} - n_i, 0}^{\mathbb{M}_{\epsilon + 1}}(\alpha_{n_i}', \ldots, \alpha_{n_{i + 1} - 1}'\})$.  Note that $\gamma_i \leq \overline{\gamma}$ for each $i < \omega$ (since $u' \subseteq  \cl_{\mathbb{M}_{\epsilon}}(\overline{\gamma})$).

Notice as well that if $i < j$ then $\gamma_i \neq \gamma_j$, since otherwise we have

$$
\begin{array}{ll}
 \alpha_{n_{i}}' & \in  \hspace{.05cm} \cl_{\mathbb{M}}(\{\alpha_{n_i + 1}', \ldots, \alpha_{n_{i + 1} - 1}'\} \cup \{\gamma_i\}) \\ & \subseteq \cl_{\mathbb{M}}(\{\alpha_{n_i + 1}', \ldots, \alpha_{n_{i + 1} - 1}'\} \cup \{\gamma_j\})\\
&   \subseteq \cl_{\mathbb{M}}(\{\alpha_{n_i + 1}', \ldots, \alpha_{n_{i + 1} - 1}'\} \cup \{\alpha_{n_j}', \ldots, \alpha_{n_{j + 1} - 1}'\})\\
& \subseteq \cl_{\mathbb{M}}(\{\alpha_j'\}_{j > n_i}).
\end{array}
$$

\noindent Also we see that $\gamma_i \notin \cl_{\mathbb{M}}(\{\gamma_j\}_{j > i})$, as otherwise we similarly obtain

$$
\begin{array}{ll}
 \alpha_{n_{i}}' & \in  \hspace{.05cm} \cl_{\mathbb{M}}(\{\alpha_{n_i + 1}', \ldots, \alpha_{n_{i + 1} - 1}'\} \cup \{\gamma_i\}) \\ & \subseteq \cl_{\mathbb{M}}(\{\alpha_{n_i + 1}', \ldots, \alpha_{n_{i + 1} - 1}'\} \cup \{\gamma_j\}_{j > i})\\
& \subseteq \cl_{\mathbb{M}}(\{\alpha_j'\}_{j > n_i}).
\end{array}
$$

Then $\sup\{\gamma_i\}_{i < \omega} \leq \overline{\gamma} < \sup\{\alpha_n\}_{n < \omega}$ and this contradicts the minimality of the ordinal $\sup\{\alpha_n\}_{n < \omega}$.
\end{proof}

We end this section by pointing out the following result.

\begin{theorem}\label{groupsnoinfinitefree}  Let $\kappa$ be an infinite cardinal.  The following are equivalent.

\begin{enumerate}

\item There is a simple torsion-free group of cardinality $\kappa$ having no infinite free subset, and in which every strictly descending chain of noncyclic subgroups is finite.

\item $\kappa < \fr(\aleph_0)$

\end{enumerate}

\end{theorem}

\begin{proof}
It quite clear that (1) implies (2).  Assuming (2), we repeat the construction given in Theorem \ref{equivalence} except that we take $m = \infty$ and each $G_{\alpha} = \langle z_{\alpha} \rangle$ is infinite cyclic.  Since $\kappa < \fr(\aleph_0) = \artin(\aleph_0)$ there exists an Artinian algebra $\mathbb{M} = (\{z_{\alpha}\}_{\alpha < \kappa}, \mathcal{F})$ having countable signature.  Define the generating function $f$ as before and apply Proposition \ref{beautifulObr} to obtain a simple torsion-free group $G$ of cardinality $\kappa$.  Lemma \ref{gettingout} holds in this setting (the proof is the same, word for word).

Suppose for contradiction that $\{g_n\}_{n < \omega}$ is an infinite free subset of $G$.  For each $k < \omega$ we let $H_k = \langle \{g_n\}_{n \geq k} \rangle$.  So $\{H_k\}_{k < \omega}$ is a strictly descending chain of subgroups of $G$, and none of the $H_k$ can be cyclic (since the infinite set $\{g_n\}_{n \geq k}$ is free in $H_k$ and a cyclic group clearly cannot have an infinite free subset).  Using Lemma \ref{gettingout}, inductively conjugate this sequence as before to a strictly descending sequence $K_0 > K_1 > \cdots$ with $K_k = \langle C_k \rangle$ with $C_k = f(X_k)$ for some $X_k \subseteq \Omega$, and $X_k \not\subseteq G_{\alpha}$ for all $\alpha$.  Obtain the same contradiction.  From our proof, every strictly descending chain of noncyclic subgroups in $G$ is finite.  Thus $G$ has the desired properties and (1) holds.

\end{proof}

\end{section}

\begin{section}{Theorems \ref{failure} and \ref{constructibleuniv}} \label{CD}

Now we are ready to prove the remainder of the theorems from known results.  The second author proved the consistency of $\Fr(\aleph_{\omega}, \aleph_0)$ \cite{Sh1} and Koepke obtained the exact consistency strength.

\begin{theorem}\label{Koepke}\cite[Theorem 2.2, 4.4]{Koepke}  The following are equiconsistent.

\begin{enumerate}

\item ZFC + ``there exists a measurable cardinal''

\item ZFC + $\Fr(\aleph_{\omega}, \aleph_0)$

\end{enumerate}

\end{theorem}

\begin{proof}[Proof of Theorem \ref{failure}]  By Theorems \ref{equivalence} and \ref{Artinfree} we know ZFC + $\Fr(\aleph_{\omega}, \aleph_0)$ is equivalent to ZFC + ``there are no Artin groups of size $\aleph_{\omega}$ or higher'', so we conclude by using Theorem \ref{Koepke}.
\end{proof}

Towards Theorem \ref{constructibleuniv} we give the following definition (see, for example, \cite{Silver}).

\begin{definition}\label{Ramseynotation}

For an infinite cardinal $\kappa$ recall that $\kappa \xrightarrow[]{} (\omega)_{\lambda}^{< \omega}$ means that for every function $f: [\kappa]^{< \omega} \rightarrow \lambda$ there exists $\{\alpha_j\}_{j \in \omega} \subseteq \kappa$ such that

\begin{enumerate}[(i)]

\item $(\forall j < \omega) \alpha_j < \alpha_{j + 1}$; and

\item $(\forall m < \omega) f \upharpoonright [\{\alpha_j\}_{j \in \omega}]^m$ is constant.

\end{enumerate}

\end{definition}

The first cardinal $E_0$ for which $E_0 \xrightarrow[]{} (\omega)_2^{< \omega}$ is also known as the \emph{first Erd\H{o}s cardinal} and is inaccessible \cite[\textsection 1]{Dev}, and therefore need not exist in a universe of set theory.  It was shown by Silver (see \cite[Theorem 1.1]{Dev}) that if $\lambda < E_0$ then $E_0 \xrightarrow[]{} (\omega)_{\lambda}^{< \omega}$.  The following proposition makes clear the relationship between $E_0$ and Artinian algebras.

\begin{proposition}\label{Ramseyfreedom}  For each infinite cardinal $\lambda < E_0$ we have $\artin(\lambda) = \fr(\lambda) \leq E_0$.
\end{proposition}

\begin{proof}  Assume the hypotheses.  By Silver's result we know $E_0 \xrightarrow[]{} (\omega)_{\lambda}^{< \omega}$.  For convenience write $\kappa = E_0$.  Suppose that $\mathbb{M} = (\kappa, \mathcal{F})$ is an algebra, with $|\mathcal{F}| \leq \lambda$.  For a finite subset $u \subseteq \kappa$ we let $\{\gamma_{u, \zeta}\}_{\zeta < \ell_u}$ be an enumeration, without repetition, of the elements of the generated subalgebra $\cl_{\mathbb{M}}(u) \subseteq \kappa$, with $\ell_u = |\cl_{\mathbb{M}}(u)| \leq \lambda$.

Let $f: [\kappa]^{< \omega} \rightarrow [\omega \times \lambda]^{<\omega}$ be defined as follows.  If $u = \{\beta_0, \ldots, \beta_{|u| - 1}\}$ where $\beta_j < \beta_{j+1}$ then we let $f(u)$ be the set of all ordered pairs $(i, \zeta)$ such that $i < |u|$, $\zeta < |\cl_{\mathbb{M}}(u \setminus \{\beta_i\})|$ and $\beta_i = \gamma_{u \setminus \{\beta_i\}, \zeta}$.  Note that $|f(u)| \leq |u|$ since the enumeration $\{\gamma_{u, \zeta}\}_{\zeta < \ell_u}$ has no repetition.  Now let $\{\alpha_j\}_{j < \omega}$ be a sequence satisfying (i) and (ii) of Definition \ref{Ramseynotation} for the function $f$ (we are using the fact that $|[\omega \times \lambda]^{< \omega}| = \lambda$).  Let $p \in \omega$ be given and suppose for contradiction that $\alpha_p \in \cl_{\mathbb{M}}(\{\alpha_j\}_{j \in \omega \setminus\{p\}})$.  Select a finite $u' \subseteq \{\alpha_j\}_{j \in \omega \setminus\{p\}}$ with $\alpha_p \in \cl_{\mathbb{M}}(u')$ and let $u = u' \cup \{\alpha_p\}$.  Write $$u = \{\alpha_{j_0}, \alpha_{j_1}, \ldots, \alpha_{j_{i-1}}, \alpha_{j_i}, \alpha_{j_{i+1}}, \ldots, \alpha_{j_{|u| - 1}}\}$$ with $\alpha_{j_s} < \alpha_{j_{s + 1}}$ for all $0 \leq s < |u| - 1$ and $p = j_i$.  Let $\zeta < \lambda$ be such that $\gamma_{u', \zeta} = \alpha_p$.  Then $(i, \zeta) \in f(u)$. Let $$\overline{u} = \{\alpha_0, \ldots, \alpha_{i}, \alpha_{i + 2}, \alpha_{i + 3} \ldots, \alpha_{|u|}\}$$ and $$\underline{u} = \{\alpha_0, \ldots, \alpha_{i - 1}, \alpha_{i + 1}, \alpha_{i + 2}, \ldots, \alpha_{|u|}\}.$$  As $f \upharpoonright [\{\alpha_j\}_{j \in \omega}]^{|u|}$ is constant we have $(i, \zeta) \in f(u) = f(\overline{u}) = f(\underline{u})$.  In particular $$\alpha_i = \gamma_{\{\alpha_0, \ldots, \alpha_{i - 1}, \alpha_{i + 2}, \ldots, \alpha_{|u|}\}, \zeta} = \alpha_{i + 1}$$ which contradicts $\alpha_i < \alpha_{i + 1}$.  Thus $\{\alpha_j\}_{j < \omega}$ is free , and so $\artin(\lambda) = \fr(\lambda) \leq E_0$.

\end{proof}
It is well to note that if $\kappa \xrightarrow[]{} (\omega)_2^{< \omega}$ then in fact $\kappa \xrightarrow[]{} (\omega)_2^{< \omega}$ in the constructible universe \cite{Silver} (for background on the constructible universe $L$  see \cite[\textsection 13]{Jech}).  Thus the existence of an Erd\H{o}s cardinal in general implies that one also exists in $L$.  Devlin and Paris give the following.

\begin{theorem}\label{DevlinwithParis}\cite[Theorem 2]{DevPar}  Assume $V = L$.  For each cardinal $\kappa$ the following are equivalent.

\begin{enumerate}

\item $\kappa \xrightarrow[]{} (\omega)_2^{< \omega}$

\item $\kappa \geq \fr(\aleph_0)$

\end{enumerate}

\end{theorem}

Using Theorems \ref{equivalence}, \ref{Artinfree} and \ref{DevlinwithParis} we obtain the following.

\begin{theorem}\label{constructiblecutoff}  Assume $V = L$. For each infinite cardinal $\kappa$ the following are equivalent.

\begin{enumerate}

\item $\kappa \xrightarrow[]{} (\omega)_2^{< \omega}$

\item There is no Artinian group of cardinality $\kappa$.

\end{enumerate}

\end{theorem}

\begin{proof}[Proof of Theorem \ref{constructibleuniv}]  We know that there are Artinian groups of each finite nonzero cardinality, since finite groups are Artinian.  If $V = L$ and there is no Erd\H{o}s cardinal then for each infinite cardinal $\kappa$ there exists an Artinian group of cardinality $\kappa$, by Theorem \ref{constructiblecutoff} and Theorem \ref{equivalence}.  On the other hand, in case the first Erd\H{o}s cardinal $E_0$ exists in $L$, notice that $(L_{E_0}, \in)$ is a model of ZFC + $V = L$, since $E_0$ is inaccessible.  Also, for any infinite cardinal $\kappa$ and function $f: [\kappa]^{< \omega} \rightarrow \{0, 1\}$ in $L_{E_0}$ we have $f \in L_{\kappa^+}$.  Thus if $\kappa \xrightarrow[]{} (\omega)_2^{< \omega}$ in $L_{E_0}$ for some cardinal $\kappa \in L_{E_0}$, then $\kappa \xrightarrow[]{} (\omega)_2^{< \omega}$ in $L$, contradicting the minimality of $E_0$.  So $(L_{E_0}, \in)$ is a model of ZFC + ($V = L$) in which there is no Erd\H{o}s cardinal and we are in the first case.
\end{proof}

\end{section}

\end{document}